\documentclass[10pt]{amsart}

\usepackage{fullpage} 
\usepackage{parskip} 
\usepackage{amsmath, amssymb, amsthm}
\usepackage{tikz-cd} 
\usepackage{verbatim}
\usepackage{hyperref}
\hypersetup{
	colorlinks = false,
	linkbordercolor = {white},
	citebordercolor = {white},
}
\usepackage[alphabetic]{amsrefs}
\usepackage[title]{appendix}
\usepackage{geometry}
\numberwithin{equation}{section}

\theoremstyle{definition}
\newtheorem{thm}[equation]{Theorem} 
\theoremstyle{definition}

\newtheorem{prop}[equation]{Proposition} 
\newtheorem{cor}[equation]{Corollary}
\newtheorem{lemma}[equation]{Lemma}

\theoremstyle{definition}
\newtheorem{exmp}{Example}[section]

\newcommand{\pr}[1]{\left(#1\right)}

\newcommand{\bb}[1]{\mathbb{#1}}

\newcommand\numberthis{\addtocounter{equation}{1}\tag{\theequation}}

\newcommand{\D}{\Delta}

\newcommand{\n}{\nabla}

\newcommand{\sbst}{\subseteq}

\newcommand{\bd}{\partial}

\title{Ancient caloric functions and parabolic frequency on graphs}
\author{Tang-Kai Lee}
\author{Archana Mohandas}

\address{MIT, Dept. of Math., 77 Massachusetts Avenue, Cambridge, MA 02139-4307}

\email{tangkai@mit.edu and amohanda@mit.edu}

\date{\today}

\begin{document}
	\begin{abstract}
	We study ancient solutions to discrete heat equations on some weighted graphs.
    On a graph of the form of a product with $\bb Z,$ we show that there are no non-trivial ancient solutions with polynomial growth.
    This result is parallel to the case of finite graphs, which is also discussed.
    Along the way, we prove a backward uniqueness result for solutions with appropriate decaying rate based on a monotonicity formula of parabolic frequency.
	\end{abstract}
	\maketitle
\section{{\bf Introduction}}

Growth of solutions to partial differential equations is a central topic in geometric analysis.
Function theory on a geometric object reflects important geometric properties in many different fields.
The most important examples are solutions to Laplacian and heat equations. Ancient caloric functions, that is, solutions to the heat equation that are defined for all negative time, and harmonic functions of polynomial and exponential growth have been extensively studied both on manifolds and on graphs.
An incomplete list of references about the related problems are provided \cite{CM97, L97, H11, CM21, H22}.

The behavior of solutions to these partial differential equations changes drastically when we replace complete manifolds with domains in a Euclidean space.
For example, following the work of Yau \cite{Y75} and Cheng \cite{C80} on manifolds with non-negative Ricci curvature, harmonic functions on such manifolds have been studied extensively. 
In fact, it is known that many manifolds with non-negative Ricci curvature admit non-trivial harmonic functions of polynomial growth, and furthermore, the space of harmonic functions of a specific polynomial growth is finite dimensional (cf. \cite{LT89, L97, CM97, CM98, H20, H23}.
However, for domains in Euclidean spaces, there are examples that only admit non-trivial harmonic functions of exponential growth.
Hang-Lin \cite{HL} proved this for infinite strips (with the Dirichlet boundary condition) for a large class of elliptic equations, and this was generalized to the parabolic case by Gui \cite{G}.

In this note, we study the discrete analogue following Hang-Lin's and Gui's ideas.
Let $G_0$ be a weighted graph\footnote{In this note, all the graphs are assumed to be locally finite, connected, and undirected.}, and $G:=G_0\Box \bb Z$ be the Cartesian product of $G_0$ and $\bb Z$ with the standard graph structure.
Such a graph $G$ is called a {\bf strip-type graph}. 
(For the definitions of these terminologies, we refer the readers to Section \ref{sec:2} of \cite{Gr}.)
Then, we show that with an appropriate Dirichlet-type condition, there is no ancient caloric function of polynomial growth on $G.$

\begin{thm}\label{thm:main-strip}
    Let $G = (V, E, w)$ be a strip-type graph with the vertex set $V = V_0\times \bb Z$
    and let $W=W_0\times \bb Z\sbst V$ where $W_0$ is a finite proper non-empty subset of $V_0.$
    Suppose $u\colon V\times \bb R_{\le 0}\to \bb R$ is an ancient caloric function supported on $W;$ that is, it satisfies
    \begin{align}\label{HE}
    \begin{cases}
        \bd_t u(x,t) = \D u(x,t)
        + c(x) u(x,t)
        &\text{for }
        (x,t)\in W\times \bb R_{\le 0}
        \\
        u(x,t) = 0 
        &\text{for }
        (x,t)\in \pr{V\setminus W}\times \bb R_{\le 0}
    \end{cases}.
    \end{align}
    Then there exists $\varepsilon_0 = \varepsilon_0(G) > 0$ such that the following holds.
    If $|c|\le \varepsilon_0$ and $u$ is of polynomial growth, in the sense that
    \begin{align*}
        |u(x,t)| \le C\pr{1 + |d(x_0, x)| + \sqrt{-t}}^d
    \end{align*}
     for some $x_0\in V,$ and positive constants $C$ and $d,$ then $u$ vanishes on the whole $V.$
\end{thm}

This can be regarded as a discrete analog of Gui's result \cite{G}.
As a corollary, Theorem~\ref{thm:main-strip} implies that such a strip-type proper subset $W$ does not admit any non-trivial harmonic functions of polynomial growth.
We remark that an assumption on $u$ is a Dirichlet-type boundary condition, and the assumption that $W_0$ is proper is necessary.
Indeed, in the following example where $W_0$ is not a proper subset of $V_0$, Theorem~\ref{thm:main-strip} does not hold.

\begin{exmp}
    Let $G = (V, E, w)$ be a weighted graph with its vertex set $V = \{v_0\} \times \bb Z,$ and its edge set $E = \{\pr{(v_0,i), (v_0, i+1)}\}$. 
    We assume all the edges have the same weight $1.$
    Let $u(x, t)$ be an ancient caloric function supported on $V$ where $x = ( v_0, i)$ for $i \in \mathbb{Z}$.
    Suppose $u(x, t)$ only depends on $x$ and hence $\partial_t u = 0$. 
    Then direct calculations imply that $u(v_0,i,t)=i$ is a (static) solution.
    Hence, the equation admits a solution with polynomial growth. 
\end{exmp}

There are many recent results concerning harmonic functions and ancient caloric functions on graphs (cf. \cite{K10, ST, HJL, HJ, MPTY, LS} or more references therein).
One such result by Hua \cite{H22} is that for graphs of polynomial volume growth; the dimension of the space of ancient solutions of polynomial growth is bounded by the growth degree and the dimension of harmonic functions with the same growth. 
In the present note, Theorem \ref{thm:main-strip} can be viewed as the first step to study the existence of harmonic functions of polynomial growth on infinite graphs.
In the finite case, the result can be proved by standard eigenvalue analysis.
We also deal with this finite case below.

\begin{thm}\label{thm:main-finite}
    Let $G = (V, E, w)$ be a finite graph.
    Suppose $u\colon V\times \bb R_{\le 0}\to \bb R$ is an ancient caloric function on $G;$ that is, it satisfies
    $$
    \bd_t u(x, t) = \D u(x,t)
    $$
    for all $(x,t)\in V\times\bb R_{\le 0}.$
    If $u$ is of polynomial growth, in the sense that
    \begin{align*}
        |u(x,t)| \le C\pr{1 + |d(x_0, x)| + \sqrt{-t}}^d
    \end{align*}
     for some $x_0\in V,$ and positive constants $C$ and $d,$ then $u$ is given by a harmonic function on $G$ and is constant in time.
\end{thm}

The proof of Theorem \ref{thm:main-finite} uses a backward uniqueness result,
which holds on an infinite graph when the solution has a growth rate bound.
This may be of independent interest in the study of the heat equation on a graph.

\begin{thm} \label{thm:backwards_uniqueness}
    Let $G = (V, E, w)$ be a weighted graph and $u\colon V\times[a,b]\to\bb R$ be a function with $u,\partial_t u\in W^{1,2}(G)$  when restricted to each time slice. 
    Suppose $\bd_t u = \D u + c(x) u (x, t)$ for some bounded $c\colon V\to \bb R.$ 
    If $u(\cdot,b)\equiv 0,$ then $u\equiv 0.$
\end{thm}

The way we prove this backward uniqueness is via the analysis of the {\bf parabolic frequency}.
The concept of frequency was employed to study the growth rate of the solution to elliptic and parabolic partial differential equations on Euclidean spaces by Almgren \cite{A} and Poon \cite{Poon} (cf. \cite{GL86}). 
Recently, Colding-Minicozzi \cite{CM22} studied a parabolic frequency on general manifolds. 
They proved a monotonicity formula of the frequency and used it to derive uniqueness properties for the heat equation on manifolds. 
The method was then explored in the context of geometric flows \cite{BK, BHL}.
In this note, we follow their ideas and investigate a parabolic frequency on general weighted graphs.
Accordingly, we obtain the backward uniqueness and the unique continuation property at infinity for the heat equation on graphs.
We remark that the uniqueness properties mentioned above do not hold for the heat equation with discrete time on weighted graphs. 
It is straightforward to construct a non-trivial solution, for example, on $K_2,$ such that it vanishes after finite time.

Although we only record these uniqueness properties for the solution to the heat equation, we should mention that the same properties hold for more general operators, since we can derive the same type of estimate for the Dirichlet energy under the assumption of theorem a stronger assumption (cf. Theorem~\ref{thm:main-freq-refined}).  
Finally, we remark that the uniqueness properties mentioned above do not hold for the heat equation with discrete time on weighted graphs. 
It is easy to construct a non-trivial solution, for example on $K_2,$ such that it vanishes after finite time.

\subsection*{\bf Acknowledgment}
The authors are grateful to Prof. Bill Minicozzi for helpful comments and to Prof. Toby Colding for his inspiring topic course on heat equations.
The paper was mainly based on the work completed in summer 2023.
During the project, Lee was partially supported by NSF Grant DMS 2005345.

\section{\bf Preliminary}\label{sec:2}
We start by introducing the terminologies about a graph.
A weighted graph $G$ consists of a vertex set $V,$ an edge set $E,$ and a weight function $w.$
The edge set is a subset of $V\times V,$ and the weight function $w$ is a function on $E$ such that $w_{xy}:= w(x,y)>0.$
For $x,y \in V,$ we say $x$ is adjacent to $y$ if $(x,y)$ and $(y,x)$ are in $E,$ denoted as $x\sim y.$

In this note, all the graphs are undirected, locally finite, and connected.
That means, $(x,y)\in E$ if and only if $(y,x)\in E$ and $w_{xy}=w_{yx}$ for $(x,y)\in E;$ for any $x\in V,$ the set $\{y\in V: y\sim x\}$ is finite; given any $x,y\in V,$ we can find $v_2,\cdots, v_{n-1}\in V$ such that $x\sim v_1\sim\cdots\sim v_{n-1}\sim y.$
For any two points $x$ and $y$ in such a graph, we define the distance between them to be
\begin{align*}
    d(x,y)
    := \begin{cases}
        0& \text{if }x=y\\
        1& \text{if }x\sim y\\
        \inf\{
    n+2:
    x\sim v_0\sim v_1\sim\cdots\sim v_{n}\sim y
    \text{ for some }v_i\in V
    \}
    &\text{otherwise}
    \end{cases}.
\end{align*}
Using this, we define the diameter of a non-trivial subset $W$ of $V$ to be
\begin{align*}
    {\rm diam} W
    := \sup\{
    d(x,y): x,y\in W
    \}.
\end{align*}
This is allowed to be infinite, and is a finite number when $W$ is a finite set.

For any $x\in V,$ the vertex weight of $x$ is 
$
\mu(x) := \sum_{y\sim x} w_{xy}.
$
For a function $f$ on $V,$ the Laplacian operator is defined as
\[
\Delta f(x) := \sum_{y \sim x}\frac{w_{xy}}{\mu(x)}(f(y)-f(x)),
\]
for any $x\in V,$ and the difference operator $\n_{xy}$ is defined as
\begin{align*}
    \nabla_{xy}f = 
    \begin{cases}
        f(y) - f(x)&\text{if }y\sim x\\
        0&\text{otherwise}
    \end{cases}
\end{align*}
for $x,y\in V.$
A function $u(x,t)$ defined on $V \times \mathbb{R}$ satisfies the heat equation if 
\[
    \frac{\partial}{\partial t} u(x, t) = \Delta u(x, t). 
\]
For the remainder of the paper (especially in Section~\ref{sec:3}), every instance of $u(x)$ is also a function of $t,$ i.e. $u(x) = u(x,t)$, unless otherwise specified.

We will consider product graphs which are defined as follows.  
Let $G_1 = (V_1, E_1, w^1)$ and $G_2 = (V_2, E_2,  w^2)$ be two weighted graphs.
Given positive numbers $p$ and $q,$ the weighted Cartesian product $G_1\Box_{p,q} G_2$ is defined to be the weighted graph $(V, E, w)$ with the following properties.
The set of vertices is $V := V_1 \times V_2.$
For $(x_1, x_2), (y_1, y_2)\in V,$ they form an edge if and only if 
\begin{align*}
\text{either }
x_1 = y_1\text{ and }x_2\sim y_2, \text{ or }
x_1\sim y_1\text{ and }x_2 = y_2.
\end{align*}
Specifically, the weight is defined by
\begin{align*}
w_{(x_1, x_2)(y_1, y_2)}
:=\begin{cases}
p\cdot \mu^1(x_1)\cdot w_{x_2 y_2}
&\text{if }x_1 = y_1\\
q\cdot \mu^2(x_2)\cdot w_{x_1 y_1}
&\text{if }x_2 = y_2\\
0
&\text{otherwise}
\end{cases}
\end{align*}
according to \cite{Gr}, where $\mu^i$ is the vertex weight of $G_i$ for $i=1,2.$
This implies that the vertex weight $\mu$ of $G$ satisfies
\begin{align}\label{weight_function}
\mu(x_1, x_2) = (p + q) \mu^1(x_1) \mu^2(x_2).
\end{align}
For the remainder of the paper, we let $p = q = 1$.
The conclusion is still true for general $p$ and $q.$

In the proof of the reverse Poincar\'e inequality in Section \ref{sec:3}, we will use Green's formula on graphs, as proved in \cite[Theorem 2.1]{Gr}.

\begin{prop}\label{Green's theorem}
    (Green's formula) Let $G = (V, E, w)$ be a weighted graph, and let $\Omega$ be a non-empty finite subset of $V$. Then, for any two functions $f, g$ on $V$,
    \begin{align} \label{Green's formula}
        \sum_{x \in \Omega} \Delta f(x)g(x)\mu(x) = -\frac{1}{2}\sum_{x, y\in \Omega}(\nabla_{xy}f)(\nabla_{xy}g)w_{xy} 
        + \sum_{x \in \Omega}\sum_{y \in \Omega^c} (\nabla_{xy}f)g(x)w_{xy}.
    \end{align}
    If $V$ is finite and $\Omega = V$, then $\Omega^c$ is empty so that the "boundary" term in \eqref{Green's formula} vanishes, and we obtain 
    \begin{align*}
        \sum_{x \in V}\Delta f(x)g(x)\mu(x) & = -\frac{1}{2}\sum_{x,y \in V} (\nabla_{xy}f)(\nabla_{xy}g)\mu_{xy}. 
    \end{align*}
\end{prop}

We will also need a different version of Green's formula when $\Omega$ is an infinite subset but one of the function has compact support.

\begin{prop}\label{prop:Green-infinite}
    Let $G = (V, E, w)$ be a weighted graph, and let $\Omega$ be a non-empty subset of $V$. Then, for any function $f$ on $V$ and any function $g$ finitely supported on $V$,
    \begin{align} \label{Greens-formula-infinite}
        \sum_{x \in \Omega} \Delta f(x)g(x)\mu(x) = -\frac{1}{2}\sum_{x, y\in \Omega}(\nabla_{xy}f)(\nabla_{xy}g)w_{xy} 
        + \sum_{x \in \Omega}\sum_{y \in \Omega^c} (\nabla_{xy}f)g(x)w_{xy}.
    \end{align}
\end{prop}
\begin{proof}
Note that
    \begin{align}  \label{eq:sum_over_omega}
        \sum_{x \in \Omega} \Delta f(x)g(x)\mu(x) & = \sum_{x \in \Omega} \left(\frac{1}{\mu(x)}\sum_{y \in V}(f(y)-f(x))\mu_{xy}\right) g(x) \mu(x) \nonumber\\
        & = \sum_{x \in \Omega} \sum_{y \in V}(f(y)-f(x))g(x) \mu_{xy} \nonumber \\
        & = \sum_{x \in \Omega} \sum_{y \in \Omega} (f(y)-f(x))g(x) \mu_{xy} + \sum_{x \in \Omega} \sum_{y \in \Omega^c} (f(y)-f(x))g(x) \mu_{xy}.
    \end{align}
    Note that since $g$ is compactly supported on $\Omega$, both terms on the right side of \eqref{eq:sum_over_omega} converge. We then obtain,
    \begin{align} \label{eq:greens_switch}
        \sum_{x \in \Omega} \Delta f(x)g(x)\mu(x) &= \sum_{y \in \Omega} \sum_{x \in \Omega} (f(x)-f(y))g(y) \mu_{xy} + \sum_{x \in \Omega} \sum_{y \in \Omega^c} (\nabla_{xy}f)g(x) \mu_{xy}.
    \end{align}
    where in \eqref{eq:greens_switch} we have switched $x$ and $y$. Adding together the identities \eqref{eq:greens_switch} and \eqref{eq:sum_over_omega}, we have
    \begin{align}
        \sum_{x \in \Omega} \Delta f(x)g(x)\mu(x) &= \frac{1}{2}\sum_{x,y \in \Omega} (f(y)-f(x))(g(x)-g(y))\mu_{xy} \nonumber \\ 
        &+ \sum_{x \in \Omega}\sum_{y \in \Omega^c}(\nabla_{xy}f)g(x)\mu_{xy} \nonumber \\
        &= -\frac{1}{2}\sum_{x, y\in \Omega}(\nabla_{xy}f)(\nabla_{xy}g)w_{xy} 
        + \sum_{x \in \Omega}\sum_{y \in \Omega^c} (\nabla_{xy}f)g(x)w_{xy}. \nonumber
    \end{align}
\end{proof}

\section{\bf Ancient caloric functions on strip-type graphs}
\label{sec:3}
\subsection{Reverse Poincar\'e inequality} 
We will need a Poincar\'e-type inequality on a graph in the proof of the reverse Poincar\'e inequality.
This is a direct consequence of the estimate of the Dirichlet eigenvalue.

\begin{lemma}\label{lem:Poincare}
    Let $G = (V, E, w)$ be a locally finite connected weighted graph, and $W$ be a finite proper subset of $V.$
    Then, there exists $C = C(W)>0$ such that for any function $f\colon V\to\bb R$ supported on $W,$ we have
    \begin{align*}
        \sum_{x\in V} f^2(x) \mu(x)
        \le C \sum_{x, y \in V} |\n_{xy} f|^2 w_{xy}.
    \end{align*}
\end{lemma}

Note that Lemma \ref{lem:Poincare} does not hold on single vertex graphs. 
In the lemma, $W$ must be a finite proper subset of the vertex set $V$. 

\begin{proof}
    By \cite[Theorem 4.5]{Gr}, the first Dirichlet eigenvalue 
    \begin{align*}
        \inf_{f} \frac{\sum_{x,y \in V} |\n_{xy} f|^2 w_{xy}}
        {\sum_{x\in V} f^2(x) \mu(x)}
    \end{align*}
    is positive\footnote{Note that in \cite{Gr}, $G$ is assumed to be infinite, which is used to make sure that $V\setminus W$ is non-empty. We already include this in the assumption of the lemma.}, where the infimum is taken over all functions $f\colon V\to\bb R$ supported on $W.$
    This positivity implies the lemma.
\end{proof}

Let $G_0 = (V_0, E_0, w_0)$ be a locally finite connected weighted graph,
and $G := G_0 \Box \bb Z$ be the product of $G_0$ and the standard graph $(\bb Z, E_e),$
where for $i,j\in \bb Z,$ $(i,j)\in E_e$ if and only if $|i-j|=1.$
For $r\in\bb N,$ let 
$D_r := V_0 \times \pr{[-r, r]\cap \bb Z}.$
For a function $w$ defined on $Q_r := [-r^2, 0] \times D_r$ for some $r\in\bb N,$ we write 
\begin{align}\label{integral_notation}
    \int_{Q_r} w
    := \int_{-r^2}^0 \sum_{x\in D_r} u(x, t) \mu(x)
    dt.
\end{align}

\begin{prop}
    There exists $\varepsilon>0$ such that the following holds.
    Suppose $c\colon V\to\bb R$ is a function satisfying $c\le \varepsilon.$ 
    If $u\colon V\times \bb R_{\le 0}\to \bb R$ is a solution to \eqref{HE}, then for integers $R > r > 0,$ we have
    \begin{align*}
        \int_{Q_r} u^2
        \le \frac{C}{(R-r)^2}
        \pr{
        \int_{Q_R} u^2 
        - \int_{Q_r} u^2
        }.
    \end{align*}
\end{prop}

\begin{proof}
    Let $\varphi$ be a compactly supported cutoff function on $V\times \bb R_{\le 0}$ 
    such that $\varphi(x,\cdot)$ is smooth for any $x\in V.$
    We will specify this cutoff function later.
    Taking the partial derivative with respect to time, we have
    \begin{align}\label{dtu}
        &\quad\frac{1}{2} \partial_t \sum_{x \in V} u^2(x) \varphi^2 (x) \mu(x)\\
        & = \frac{1}{2} \sum_{x \in V} (2u(x) \partial_t u(x)  \varphi^2 (x) + 2\varphi (x) \partial_t \varphi (x) u^2(x)) \mu(x) \nonumber \\ 
        & = \sum_{x \in V}  u(x) \Delta u(x) \varphi^2 (x) \mu(x) + \sum_{x \in V}  cu^2(x)\varphi^2 (x) \mu(x) 
        + \sum_{x \in V} u^2(x) \varphi (x) \partial_t \varphi (x) \mu(x). \nonumber
    \end{align}    
    Applying Proposition \ref{prop:Green-infinite} to the first term in \eqref{dtu} with $\Omega = W,$ 
    \begin{align} \label{eq:apply-greens}
        \sum_{x \in V}  u(x) \Delta u(x) \varphi^2(x) \mu(x)
        & = \sum_{x \in  W}  u(x) \Delta u(x) \varphi^2(x) \mu(x)
        \\
        & = -\frac{1}{2} \sum_{x,y \in W}\nabla_{xy}u\nabla_{xy}(u\varphi^2)w_{xy} 
        + \sum_{x \in W}\sum_{y \in V \setminus W}(\nabla_{xy}u)u(x)\varphi^2(x)w_{xy}. \nonumber
    \end{align}
    Since $u(y) = 0$ for $y \in V \setminus W$, the last term in the above equation becomes,
    \begin{align}
        \sum_{x \in W}\sum_{y \in V\setminus W}(\nabla_{xy}u)u(x)\varphi^2(x)w_{xy} \nonumber
        & = \sum_{x \in W}\sum_{y \in V \setminus W}
        (u(y)-u(x))u(x)\varphi^2(x)w_{xy} \nonumber \\
        & = - \sum_{x \in W}\sum_{y \in V \setminus W} u^2(x)\varphi^2(x)w_{xy} \nonumber \\
        & \leq 0. \nonumber
    \end{align}
    Dropping this term from \eqref{eq:apply-greens}, we obtain,
    \begin{align}\label{before_swap}
        &\quad\sum_{x \in V}  u(x) \Delta u(x) \varphi^2(x) \mu(x) \nonumber\\
        & \leq -\frac{1}{2} \sum_{x,y \in W}\nabla_{xy}u\nabla_{xy}(u\varphi^2)w_{xy} \nonumber\\
        & = -\frac{1}{2} \sum_{x,y\in W} \nabla_{xy}u \pr{u(y)\nabla_{xy}\varphi^2 + \varphi^2(x)\nabla_{xy}u} w_{xy} \nonumber\\  
        & = -\frac{1}{2}\sum_{x,y \in W}u(y)\nabla_{xy}u(\varphi(x)\nabla_{xy}\varphi+\varphi(y)\nabla_{xy}\varphi)w_{xy}
        -\frac{1}{2} \sum_{x,y \in W} |\nabla_{xy}u|^2\varphi^2(x) w_{xy} \nonumber \\
        & = -\frac{1}{2}\sum_{x,y \in W}u(y)\nabla_{xy}u\nabla_{xy}\varphi(2\varphi(x) + \nabla_{xy}\varphi)w_{xy} 
        -\frac{1}{2} \sum_{x,y \in W} |\nabla_{xy}u|^2\varphi^2(x) w_{xy} \nonumber \\
        & = -\sum_{x,y \in W}\varphi(x)u(y)\nabla_{xy}u \nabla_{xy}\varphi ~w_{xy} - \frac{1}{2}\sum_{x,y \in W}u(y)\nabla_{xy}u|\nabla_{xy}\varphi|^2 w_{xy}
        -\frac{1}{2} \sum_{x,y \in W} |\nabla_{xy}u|^2\varphi^2(x) w_{xy}.\nonumber 
    \end{align}
    For the second term, by swapping $x$ and $y$, the symmetry yields,
    \begin{align*}
        - \frac{1}{2}\sum_{x,y \in W}u(y)\nabla_{xy}u|\nabla_{xy}\varphi|^2 w_{xy}
        & = -\frac{1}{4}\sum_{x,y \in W}(u(y)-u(x))\nabla_{xy}u|\nabla_{xy}\varphi|^2 w_{xy}   \\
        & = -\frac{1}{4}\sum_{x,y \in W}|\nabla_{xy}u|^2|\nabla_{xy}\varphi|^2 w_{xy} 
        \leq 0 .
    \end{align*}
    Dropping this non-positive term and using $\varepsilon \geq c$, from \eqref{dtu}, we have,
    \begin{align}
        \frac{1}{2} \partial_t \sum_{x \in V} u^2(x) \varphi^2 (x) \mu(x) 
        & \leq -\sum_{x,y \in W} \varphi(x) u(y) \nabla_{xy}u \nabla_{xy}\varphi w_{xy}
        - \frac{1}{2} \sum_{x,y \in W} |\nabla_{xy}u|^2\varphi^2(x) w_{xy} \\
        & + \varepsilon \sum_{x \in V}  u^2(x)\varphi^2(x) \mu(x) 
        + \sum_{x \in V} u^2(x) \varphi (x) \partial_t \varphi (x) \mu(x). \nonumber
    \end{align}
    Using the AM-GM inequality for the first term on the right side of the inequality, we obtain
    \begin{align} \label{2ab_ineq}
        \frac{1}{2} \partial_t \sum_{x \in V} u^2(x) \varphi^2 (x) \mu(x) 
        & \leq \sum_{x,y \in W}\frac{1}{4} \varphi^2(x)|\nabla_{xy}u|^2 w_{xy} 
        + \sum_{x,y\in W} u^2(y) |\nabla_{xy}\varphi|^2 w_{xy} \\
        &\quad -\frac{1}{2} \sum_{x,y\in W} |\nabla_{xy}u|^2 \varphi^2(x) w_{xy} 
        + \varepsilon \sum_{x \in V}  u^2(x) \varphi^2 (x) \mu(x)   \nonumber\\
        &\quad + \sum_{x\in V} u^2(x) \varphi (x) \partial_t \varphi (x) \mu(x). \nonumber\\
        & = -\frac{1}{4} \sum_{x,y \in W} \varphi^2(x) |\nabla_{xy}u|^2 w_{xy} 
        + \sum_{x,y \in W} u^2(y) |\nabla_{xy}\varphi|^2 w_{xy} \nonumber\\
        &\quad + \varepsilon \sum_{x\in V}  u^2(x) \varphi^2(x) \mu(x) 
        + \sum_{x \in V} u^2(x) \varphi(x) \partial_t \varphi (x) \mu(x).\nonumber
    \end{align}
    Rewriting the weight function $\mu_G(x_0, z)$ on the product graph (with the vertex set $V_0\times\bb Z$) in terms of the weight function $\mu^0(x_0)$ on $G_0$ (with the vertex set $V_0$) we have,
    \begin{align}\label{mu0}
        \mu(x_0, z)
        & = \sum_{(x_0', z')\sim (x_0, z)} w_{(x_0', z)(x_0, z')}\\
        & =\sum_{x_0'\sim x_0} w_{(x_0', z)(x_0, z)}
        +  w_{(x_0, z)(x_0, z+1)}
        + w_{(x_0, z)(x_0, z-1)} \nonumber\\
        & = \sum_{x_0'\sim x_0} 2 w^0_{x_0'x_0}
        + \mu^0(x_0)
        + \mu^0(x_0) \nonumber \\
        & = 4\mu^0(x_0).\nonumber
    \end{align}
    We will take $\varphi$ independent of $x_0\in V_0$ later (see \eqref{construction-varphi}).
    Applying Lemma \ref{lem:Poincare} to the third term on the right side of \eqref{2ab_ineq}, and substituting $\mu(x)$ with $4\mu^0(x_0)$ we obtain 
    \begin{align} 
        \label{eq:poincare}
        \varepsilon \sum_{x \in V}  u^2(x)\varphi^2 (x) \mu(x)
        & = 4\varepsilon \sum_{z \in \mathbb{Z}} \varphi^2 (z) \sum_{x_0 \in V_0} u^2(x_0,z)\mu^0(x_0) \nonumber\\ 
        & \leq 4C \varepsilon \sum_{z \in \mathbb{Z}} \varphi^2 (z) \sum_{x_0, y_0 \in V_0} |\nabla_{x_0y_0}u|^2 w^0_{x_0y_0} 
        \nonumber \\
        & \leq 4C \varepsilon \sum_{x,y \in V} \varphi^2 (x) |\nabla_{xy}u|^2w_{xy}
    \end{align}
    for a constant $C=C(W_0)>0.$
    Substituting \eqref{eq:poincare} into \eqref{2ab_ineq}, we have
    \begin{align}
        \frac{1}{2} \partial_t \sum_{x \in V} u^2(x) \varphi^2 (x) \mu(x) 
        & \leq -\frac{1}{4}\sum_{x,y \in W}\varphi^2(x)|\nabla_{xy}u|^2 w_{xy} + \sum_{x,y \in W}u^2(x)|\nabla_{xy}\varphi|^2 w_{xy} \nonumber \\
        & + 4C \varepsilon \sum_{x,y \in V} \varphi^2 (x) |\nabla_{xy}u|^2w_{xy} 
        + \sum_{x \in V} u^2(x) \varphi (x) \partial_t \varphi (x) \mu(x). \nonumber \\
        & = \pr{4C\varepsilon-\frac{1}{4}}
        \sum_{x,y \in W}\varphi^2(x)|\nabla_{xy}u|^2 w_{xy} \nonumber \\
        & + \sum_{x,y \in V}u^2(x)|\nabla_{xy}\varphi|^2 w_{xy} + \sum_{x \in V} u^2(x) \varphi (x) \partial_t \varphi (x) \mu(x). \nonumber
    \end{align}
    Rearranging terms in the above inequality,
    \begin{align}
        \partial_t \sum_{x \in V} u^2(x) \varphi^2 (x) \mu(x) + \left(\frac{1-16C\varepsilon}{2}\right)\sum_{x,y \in W}\varphi^2(x)|\nabla_{xy}u|^2 w_{xy} \nonumber \\
        \leq 2\sum_{x,y \in V}u^2(x)|\nabla_{xy}\varphi|^2 w_{xy} + 2\sum_{x \in V} u^2(x) \varphi (x) \partial_t \varphi (x) \mu(x). \nonumber
    \end{align}
    We choose $\varepsilon = \frac{1}{32C}$ so that $\frac{(1-16C\varepsilon)}{2} = \frac{1}{4} > 0$.
    Integrating over $t \in [-R^2, 0]$, we have
    \begin{align} 
        \label{before_poincare}
        &\sum_{x \in V} u^2(x) \varphi^2 (x) \mu(x) |_{t=0}
        - \sum_{x \in V} u^2(x) \varphi^2 (x) \mu(x) |_{t=-R^2}
        + \frac 14\int_{-R^2}^0  \sum_{x,y \in W}\varphi^2(x)|\nabla_{xy}u|^2 w_{xy} dt \\ 
        &\leq 2 \int_{-R^2}^0 \sum_{x,y \in V} u^2(x) |\nabla_{xy}\varphi|^2 w_{xy} dt 
        + 2 \int_{-R^2}^0 \sum_{x \in V} u^2(x) \varphi (x) \partial_t \varphi (x) \mu(x) dt. \nonumber
    \end{align}
    \noindent Applying Lemma \ref{lem:Poincare} to the third term on the left side of \eqref{before_poincare}, we obtain, 
    \begin{align} \label{after_poincare}
        & \frac{1}{4}\int_{-R^2}^0  \sum_{x,y \in W}\varphi^2(x)|\nabla_{xy}u|^2 w_{xy} dt \nonumber \\
        & \geq \frac{1}{4} \int_{-R^2}^0  \sum_{z \in \mathbb{Z}}\varphi^2(z) \sum_{x_0, y_0 \in W_0}|\nabla_{xy}u|^2 w_{x_0y_0} \nonumber dt \\
        & \geq \frac{1}{4C} \int_{-R^2}^0  \sum_{z \in \mathbb{Z}}\varphi^2(z) \sum_{x_0 \in W_0} u^2(x_0,z) \mu^0(x_0) dt.
    \end{align}
    Substituting $\frac{1}{4}\mu(x)$ for $\mu^0(x_0)$ in \eqref{after_poincare}, we have
    \begin{align}
        \label{poincare_final}
        \frac{1}{4}\int_{-R^2}^0  \sum_{x,y \in W}\varphi^2(x)|\nabla_{xy}u|^2 w_{xy} dt
        &\geq \frac{1}{16C} \int_{-R^2}^0  \sum_{x \in W} u^2(x) \varphi^2(x) \mu(x) dt.
    \end{align}
    Thus, we get
    \begin{align}\label{u-bound}
        &\int_{t=0} \sum_{x \in V} u^2(x) \varphi^2 (x) \mu(x) dt + \frac 1{16C}\int_{-R^2}^0  \sum_{x \in W} u^2(x) \varphi^2(x) \mu(x) dt \\ 
        &\leq 2 \int_{-R^2}^0 \sum_{x,y \in V} u^2(x) |\nabla_{xy}\varphi|^2 w_{xy} dt 
        + 2 \int_{-R^2}^0 \sum_{x \in V} u^2(x) \varphi (x) \partial_t \varphi (x) \mu(x) dt.\nonumber
    \end{align}

    Now, we specify the cutoff function $\varphi$. 
    We define it by 
    \begin{align}\label{construction-varphi}
        \varphi((i,x_0),t) = 
        \begin{cases}
            1 & |i| \leq r  \quad \text{and} \quad t \ge -r^2\\
            0 & |i| \geq R \quad \text{or} \quad t \le -R^2\\
            \frac{R-i}{R-r}\phi(t) &  \quad \text{otherwise}            
        \end{cases} 
    \end{align}
    where $x_0$ is a fixed point in $W_0$ and $\phi(t) \in C_c^\infty$ is a compactly supported smooth function in the time direction such that
    \begin{align*}
        \phi(t) =\begin{cases}
            1 &\text{if }|t|\le r^2\\
            0 &\text{if }|t|\ge R^2\\
        \end{cases}
    \end{align*}
    and $|\bd_t\phi|\le \frac 2{R^2-r^2}.$
    Then, $\varphi$ satisfies $\varphi = 1$ on $Q_r$ and $\varphi = 0$ outside $Q_R$ and $\varphi \in [0,1]$. 
    We can also let
    \begin{align} \label{cutoff_bounds}
        |\nabla_{xy}\varphi| \leq \frac{2}{R-r} \text{  and  } |\partial_t\varphi| \leq \frac{2}{R^2-r^2} \leq \frac{2}{(R-r)^2}. 
    \end{align}
    Combining \eqref{u-bound}, and \eqref{cutoff_bounds}, we obtain
    \begin{align*}
        \frac{1}{16C} \int_{Q_r} u^2(x) 
        \leq \frac{8}{(R-r)^2} \int_{Q_R\setminus Q_r} u^2(x).
    \end{align*}
    This completes the proof.
\end{proof}

\subsection{Proof of Theorem~\ref{thm:main-strip}}

We will now prove the first main theorem.
Recall the integral notation defined in equation \eqref{integral_notation} in Section~\ref{sec:3}. 
\begin{proof}[Proof of Theorem~\ref{thm:main-strip}]
    Fix $r > 1$. For any $R > r$, we have
    \begin{align}
        \int_{Q_r} u^2(x) \leq \frac{C}{(R-r)^2} \left( \int_{Q_R} u^2(x) - \int_{Q_r} u^2(x)\right). \nonumber
    \end{align}
    Rearranging terms, we obtain
    \begin{align}
        \int_{Q_R} u^2(x)
        \geq  \frac{\left( 1+ \frac{C}{(R-r)^2}\right)}{\frac{C}{(R-r)^2}}
        \int_{Q_r} u^2(x). \nonumber
    \end{align}
    Choose $r_0 = \sqrt{C(e-1)} > 0$. 
    Then, letting $R = r+r_0$, we have
    \begin{align}\label{inequality_iter}
        \int_{Q_{r+r_0}}u^2(x) 
        \geq e\int_{Q_r}u^2(x).
    \end{align}
    Iterations of inequality \eqref{inequality_iter} imply that for any integer $k>0$,
    \begin{align}
        \int_{Q_{r+kr_0}}u^2(x) \geq e\int_{Q_{r+(k-1)r_0}}u^2(x) \geq ... \geq e^k\int_{Q_r}u^2(x). \nonumber
    \end{align}
    If $u \in P_d$ is an ancient solution with polynomial growth, then for any $R>1,$ 
    \begin{align*}
        \int_{Q_R} u^2(x)
        & \le \int_{-R^2}^0 \sum_{x\in D_R\cap W} \pr{C\pr{1 + |d(x_0, x)| + \sqrt{-t}}^d}^2 \mu(x) dt
        &\le CR^{2d} \int_{-R^2}^0 \sum_{x\in D_R\cap W} \mu(x) dt
    \end{align*}
    where $d(x_0,x)\le d(x_0,W_0\times\{0\}) + {\rm diam }W_0\le R+C(W_0)$
    if we assume $x_0\in W_0\times\{0\}.$ 
    Using \eqref{mu0}, we get
    \begin{align*}
        \int_{Q_R} u^2(x)
        & \leq CR^{2d+2} \sum_{z=-R}^R \sum_{x_0\in W_0} 4\mu^0(x_0)\\
        & = CR^{2d+2} \sum_{z=-R}^R C(W_0)\\
        & = 2C R^{2d + 3}.
    \end{align*}
    Combining all of these implies
    \begin{align}
        \int_{Q_r}u^2(x) \leq \frac{1}{e^k} \int_{Q_{r+kr_0}} u^2(x) \leq \frac{2C(r+kr_0)^{2d+3}}{e^k}, \nonumber
    \end{align}
    which leads to a contradiction for large $k$ unless $u(x) \equiv 0$.
\end{proof}

\section{\bf Ancient caloric functions on finite graphs}

For a finite graph, we will prove a stronger result that implies Theorem~\ref{thm:main-finite} directly.
The main technique is linear algebra theory.

\begin{thm}\label{thm:main-freq-refined}
     Let $G = (V, E, w)$ be a finite connected weighted graph.
    If $u\colon V\times \bb R_{\le 0}\to \bb R$ satisfies $\bd_t u= \D u,$ then $u$ has exponential growth in time.
    In particular, it cannot be a non-trivial function with polynomical growth (in time).
\end{thm}

\begin{proof}

Let $n=|V|<\infty$ and let $A$ be the matrix corresponding to $\D$ on $G.$
By \cite[Theorem~2.7]{Gr}, for a finite, connected, weighted graph $G$, the eigenvalues of the matrix of the Laplacian are contained in $[-2,0]$. Since $\frac{d}{dt}e^{At} = Ae^{At},$
the backwards uniqueness (Theorem~\ref{thm:backwards_uniqueness}) implies that the solution to $\partial_t u  = Au$ with final condition $u(\cdot, 0) = u_0$ is given by $u = e^{At}u_0$.

Writing the matrix $A$ in Jordan normal form, we obtain
$A = QJQ^{-1}$ where $J$ is the Jordan form of $A$ and $Q$ is an invertible matrix.
Taking the exponential of $At,$ we obtain,
\begin{align} \label{exponential_A}
    e^{At} &= Q e^{Jt} Q^{-1} \nonumber \\ 
    &= Q\begin{pmatrix}
        e^{J_{\lambda_1,n_1}t} & & \\
        & \ddots & \\
        & & e^{J_{\lambda_m,n_m}}t
    \end{pmatrix}Q^{-1}.
\end{align}
where $m$ is the number of the Jordan blocks,
$\lambda_i$'s are the  eigenvalues of $A$ (including repeated eigenvalues),
and $J_{\lambda_i, j}$'s are the corresponding Jordan blocks.
Each Jordan block can be written in the form 
\begin{align} \label{jordan_matrix}
    J_{\lambda_i, n_i} = \lambda_iI_{n_i}+N_{n_i},
\end{align}
where $N_{n_i}$ is an $n_i \times n_i$ nilpotent matrix.
Substituting \eqref{jordan_matrix} into \eqref{exponential_A}, we obtain,
\begin{align} \label{exp_A}
     e^{At} &= 
     Q{\begin{pmatrix}
         e^{\lambda_1t} e^{N_{n_1}t} & & \\
         & \ddots & \\
         & & e^{\lambda_mt} e^{N_{n_m}t}
     \end{pmatrix}}Q^{-1}.\nonumber \\
     &= Q{\begin{pmatrix}
         e^{\lambda_1t}\sum_{k=0}^{n_1}\frac{1}{k!}(N_{n_1}t)^k & & \\
         & \ddots & \\
         & & e^{\lambda_mt}\sum_{k=0}^{n_m}\frac{1}{k!}(N_{n_m}t)^k
     \end{pmatrix}}Q^{-1}.
\end{align}   

For each $i,$ summing over $k$, we have
\begin{align*}
    \sum_{k=0}^{n_i}\frac{1}{k!}(N_{n_i}t)^k 
    & =  \begin{pmatrix}
        1 & t & \frac{1}{2}t^2&\hdots & \frac{1}{(n_{i}-1)!}t^{n_{i}-1} \\
        0 & 1 & t & \ddots & \vdots\\
        \vdots & & \ddots & & \frac{1}{2}t^2\\
        \vdots & & \ddots & & t\\
        0 & \hdots & \hdots &  0 & 1
    \end{pmatrix},
\end{align*} 
showing that each entry of these  submatrices has at most polynomial growth in $t$. 
Then, each entry of the following matrix 
\begin{align*}
    e^{\lambda_it}\sum_{k=0}^{n_i}\frac{1}{k!}(N_{n_i}t)^k  
\end{align*}
is exponential in $t$ if $\lambda_i \neq 0$ and constant in $t$ if $\lambda_i = 0$. 
Hence, the solution $u =  e^{At}u_0$ is a linear combination of exponential and constant functions based on the final condition $u_0$. 
It follows that if $u$ satisfies
\begin{align*}
        |u(x,t)| \le C\pr{1 + \sqrt{-t}}^d
    \end{align*}
for positive constants $C$ and $d,$ then the exponential part of $u$ vanishes and $u$ is constant. 
\end{proof}

\section{\bf Backward uniqueness}

\subsection{Frequency on Weighted Graphs}
Let $G = (V, E, w)$ be a locally finite connected weighted graph.
Given two functions $f,g: V \rightarrow \mathbb{R}$, we consider the inner product of their gradients\footnote{This is related to the difference operator $\nabla_{xy}$, but here we write it in this way to emphasize its connection with the continuous case.}
\begin{align*}
    \langle \nabla f, \nabla g \rangle (x):= \sum_{y \sim x}\frac{w_{xy}}{\mu(x)}(f(y)-f(x))(g(y)-g(x))
\end{align*}
and $|\nabla f|^2(x) := \langle \nabla f, \nabla f \rangle (x)$.
For convenience, we may sometimes denote
\begin{align*}
    \sum_{V} \langle \nabla f, \nabla g \rangle \mu(x) := \sum_{x \in V} \langle \nabla f, \nabla g \rangle (x) \mu(x).
\end{align*}
We also introduce the function space 
\begin{align*}
    W^{1,2}_{\mu}(V) := \left\{ u: V \rightarrow \mathbb{R}: \sum_{x \in V} u^2 \mu(x) < \infty \right\}.
\end{align*}
Let $G = (V, E, w)$ be a weighted graph and $[a,b]\sbst\bb R$ be a time interval. For a nontrivial function $u\colon V\times [a,b]\to\bb R$ such that $u(x,\cdot)$ is $C^1$ given any $x\in V$ and that $u,\partial_t u\in W^{1,2}_\mu(V),$ we define the following functions
\begin{equation*}
I(t) := \sum_{x\in V} u^2\mu(x),
\end{equation*}
\begin{equation*}
D(t) := -\sum_{x\in V} |\n u|^2\mu(x),
\end{equation*}
and
\begin{equation*}
U(t) := \frac{D(t)}{I(t)} = - \frac{\sum_{x\in V} |\n u|^2\mu(x)}{ \sum_{x \in V} u^2\mu(x)}.
\end{equation*}
We call $U$ the frequency function for $u.$ This is a discrete version of the one defined in \cite{CM22}
\footnote{Note that we follow the sign convention in \cite{CM22}. That is, the frequency function is always non-positive, unlike that in \cite{A} and \cite{Poon}.} 
and the methods employed there enable us to get the corresponding results on weighted graphs.

\subsection{\bf Frequency Monotonicity and Backward Uniqueness}
In this section, we will prove the main estimate and its corollaries.
\begin{thm}
\label{U'}
	Let $(V,E,w)$ be a weighted graph and $u\colon V\times[a,b]\to\bb R$ with $u,\partial_t u\in W^{1,2}_\mu(V).$ If there exists $C=C(t)>0$ such that
	\begin{equation}\label{gass}
	|(\bd_t-\D)u|\le C(t)(|u|+|\n u|),
	\end{equation}
	then we have
	\begin{equation*}
	U'\ge 2C^2(U-1).
	\end{equation*}
\end{thm}

\begin{proof}
     By definition, we have
\begin{equation*}
D
= 2\sum_{x \in V} u\D u~\mu(x)
= 2\sum_{x \in V} u\left(\partial_t u - \frac 12(\partial_t u-\D u) \right)\mu(x) 
- \sum_{x \in V} u(\partial_t u-\D u)\mu(x)
\end{equation*}
and
\begin{equation*}
I'
= 2\sum_{x \in V} u\partial_t u~\mu(x)
= 2\sum_{x \in V} u\left(\partial_t u - \frac 12(\partial_t u-\D u) \right)\mu(x) 
+ \sum_{x \in V} u(\partial_t u-\D u)\mu(x),
\end{equation*}
so
\begin{equation}\label{I'D}
I' D = 4\left(\sum_{x \in V} u\left(\partial_t u - \frac 12(\partial_t u-\D u) \right)\mu(x) \right)^2
- \left(\sum_{x \in V} u(\partial_t u-\D u)\mu(x)\right)^2.
\end{equation}
On the other hand,
\begin{align*}
D' I
= 4\sum_{x \in V} \partial_t u\D u~\mu(x)\cdot I
& = 4\sum_{x \in V} \partial_t u\left(\partial_t u-(\partial_t u-\D u)\right)\mu(x)\cdot I\\
& = 4\sum_{x \in V} \left(\left(\partial_t u-\frac 12(\partial_t u-\D u)\right)^2 - \frac 14(\partial_t u-\D u)^2 \right)\mu(x)\cdot I.
\end{align*}
Combining this with \eqref{I'D} and using the Cauchy-Schwarz inequality and the assumption \eqref{gass}, we derive
\begin{align*}
&~~~~D'I-I'D\numberthis\label{CS}\\
& = 4\left(\left(\sum_{x \in V}u^2\mu(x)\right)\left(\sum_{x \in V} \left(\partial_t u-\frac 12(\partial_t u-\D u)\right)^2\mu(x)\right)
-  \left(\sum_{x \in V} u\left(\partial_t u - \frac 12(\partial_t u-\D u) \right)\mu(x) \right)^2
\right)\\
& - I \sum_{x \in V} (\partial_t u-\D u)^2\mu(x)
+ \left(\sum_{x \in V} u(\partial_t u-\D u)\mu(x)\right)^2\\
& \ge - C^2 I \sum_{x \in V} (|u|+|\n u|)^2 \mu(x)
\ge - 2 C^2 I \sum_{x \in V} (u^2 + |\n u|^2)\mu(x)
= -2C^2 (I^2-ID).
\end{align*}
As a result, 
\begin{align*}
U'
= \frac{D'I-I'D}{I^2}
\ge -2C^2\left(1-\frac DI\right)=2C^2(U-1),
\end{align*}
and the result follows. 
\end{proof}

As a corollary, we have a Harnack-type inequality for the Dirichlet energy $I.$
\begin{cor}\label{unique}
	Let $(V,E,w)$ be a weighted graph and $u\colon V\times[a,b]\to\bb R$ with $u,\partial_t u\in W^{1,2}_\mu(V).$ If there exists $C(t)>0$ such that
	\begin{equation*}
	|(\bd_t-\D)u|\le C(t)(|u|+|\n u|),
	\end{equation*}
	then we have
	\begin{equation}\label{Iab}
	I(b)\ge I(a)\cdot e^{(b-a)\left((1+K)(U(a)-1)e^{2(b-a)K^2}-3K\right)}
	\end{equation}
	where $K:=\sup_{[a,b]}C.$
\end{cor}

Corollary \ref{unique} infers Theorem \ref{thm:backwards_uniqueness} immediately.
In fact, when $u$ satisfies $\pr{\bd_t - \D}u = cu,$ we have $K=\sup_V c<\infty.$
Thus, \eqref{Iab} implies that if $u(\cdot, b)\equiv 0,$ then $u(\cdot, t)\equiv 0$ for all $t\in[a,b].$
This proves Theorem~\ref{thm:backwards_uniqueness}.

\begin{proof}
    First, we write
$$I' 
= 2\sum_{x \in V} u\D u~\mu(x) + 2\sum_{x \in V} u(\partial_t u-\D u)\mu(x)
= D + 2\sum_{x \in V} u(\partial_t u-\D u)\mu(x),$$
and then estimate, using the Cauchy-Schwarz inequality,
\begin{align*}
(\log I)'
= \frac D I + \frac 2 I \sum_{x \in V} u(\partial_t u-\D u)\mu(x)
\ge U - \frac {2C}I\sum_{x,y \in V} |u|(|u|+|\n u|)\mu(x)
& \ge U - 2C - \frac{2C}{I} \cdot \sqrt{I} \sqrt{-D}
\end{align*}
in which
$$
U - 2C(1+\sqrt{-U})\\
\ge U - 2C\left(1+\frac {-U}2 + \frac 12\right)
\ge (1+C)U -3C,$$
so
$$(\log I)'\ge (1+K)U -3K.$$
As a result,
\begin{equation}\label{logi}
\log I(b) - \log I(a)
\ge (1+K)\int_a^b U(t) dt - 3K(b-a).
\end{equation}
To estimate the first term, note that theorem \ref{U'} implies 
\begin{equation*}
(\log(1-U))'
= \frac{U'}{U-1}
\le 2C^2.
\end{equation*}
Therefore, for $s\in[a,b],$ 
\begin{align*}
\log(1-U(s))-\log(1-U(a))
\le 2\int_a^s C^2 dt
\le 2(b-a)K^2,
\end{align*}
so
$$U(s)
\ge (U(a)-1)e^{2(b-a)K^2}+1.
$$
Plugging this into \eqref{logi}, we obtain
$$
\log I(b) - \log I(a)
\ge (1+K)(b-a)((U(a)-1)e^{2(b-a)K^2}+1)-3K(b-a).
$$
Taking the exponential of both sides then leads to the desired estimate \eqref{Iab}. 
\end{proof}

Finally we have the following equality case when $C=0.$
\begin{cor}
	Let $(V,E,w)$ be a weighted graph and $u\colon V\times[a,b]\to\bb R$ with $u,\partial_t u\in W^{1,2}_\mu(V).$ If $u$ is a solution to the heat equation, then $U'\ge 0.$ Moreover, if $U'\equiv 0,$ then $\D u = \frac U2 u$ and $u(x,t) = e^{\frac U2 (t-a)}u(x,a).$ 
\end{cor}

\begin{proof}
The first conclusion directly follows from theorem \ref{U'}. When $U'\equiv 0$ and $\bd_t u = \D u,$ the inequality \eqref{CS} is an equality. Since the inequality \eqref{CS} follows from the Cauchy-Schwarz inequality, the equality holds if and only if
$$\D u = \partial_t u=c(t)u$$
for some $c(t)\in\bb R.$ Consequently,
$$D=2\sum_{x \in V} u\D u~\mu(x)
= 2c\sum_{x \in V} u^2\mu(x) = 2cI,
$$
which implies
$$c = \frac D{2I} = \frac U2$$
is also constant. Hence 
$$\D u = \partial_t u = cu=\frac{U}{2}u.$$
Based on this, we observe that
$$\bd_t\left(e^{-\frac U2 t} u(x,t) \right)
= e^{-\frac U2 t}\left(-\frac U2 u + \bd_t u\right) =0
$$
and hence we get
$$e^{-\frac U2 t} u(x,t)=e^{-\frac U2 a} u(x,a)$$
for all $t\in[a,b],$ so the conclusion follows.
\end{proof}

\end{document}